\documentclass{amsart}
\usepackage{amsfonts}

\newtheorem{theorem}{Theorem}
\theoremstyle{plain}

\newtheorem{definition}{Definition}

\newtheorem{lemma}{Lemma}

\numberwithin{equation}{section}
\input{tcilatex}

\begin{document}
\title[On Hermite Hadamard inequalities]{On Hermite Hadamard inequalities
for product of two $\log $-$\varphi $-convex functions }
\author{Mehmet Zeki SARIKAYA}
\address{Department of Mathematics, \ Faculty of Science and Arts, D\"{u}zce
University, D\"{u}zce-TURKEY}
\email{sarikayamz@gmail.com}
\subjclass[2000]{ 26D10, 26A51,46C15}
\keywords{Hermite-Hadamard's inequalities, $\varphi $-convex functions, $%
\log $-$\varphi $-convex functions.}

\begin{abstract}
In this paper, we introduce the notion of $\log $-$\varphi $-convex
functions and present some properties and representation of such functions.
We obtain some results of the Hermite Hadamard inequalities for product $%
\log $-$\varphi $-convex functions.
\end{abstract}

\maketitle

\section{Introduction}

The inequalities discovered by C. Hermite and J. Hadamard for convex
functions are very important in the literature (see, e.g.,\cite{dragomir1},%
\cite[p.137]{pecaric}). These inequalities state that if $f:I\rightarrow 
\mathbb{R}$ is a convex function on the interval $I$ of real numbers and $%
a,b\in I$ with $a<b$, then 
\begin{equation}
f\left( \frac{a+b}{2}\right) \leq \frac{1}{b-a}\int_{a}^{b}f(x)dx\leq \frac{%
f\left( a\right) +f\left( b\right) }{2}.  \label{E1}
\end{equation}%
The inequality (\ref{E1}) has evoked the interest of many mathematicians.
Especially in the last three decades numerous generalizations, variants and
extensions of this inequality have been obtained, to mention a few, see (%
\cite{bakula}-\cite{set2}) and the references cited therein.

The function $f:[a,b]\subset \mathbb{R}\rightarrow \mathbb{R}$, is said to
be convex if the following inequality holds%
\begin{equation*}
f(\lambda x+(1-\lambda )y)\leq \lambda f(x)+(1-\lambda )f(y)
\end{equation*}%
for all $x,y\in \lbrack a,b]$ and $\lambda \in \left[ 0,1\right] .$ We say
that $f$ is concave if $(-f)$ is convex.

A function $f:I\rightarrow \lbrack 0,\infty )$ is said to be log-convex or
multiplicatively convex if $\log t$ is convex, or, equivalently, if for all $%
x,y\in I$ and $t\in \left[ 0,1\right] $ one has the inequality:%
\begin{equation}
f\left( tx+\left( 1-t\right) y\right) \leq \left[ f\left( x\right) \right]
^{t}\left[ f\left( y\right) \right] ^{1-t}.  \label{E2}
\end{equation}

We note that if $f$ and $g$ are convex and $g$ is increasing, then $g\circ f$
is convex; moreover, since $f=\exp \left( \log f\right) $, it follows that a 
$log$-convex function is convex, but the converse may not necessarily be
true \cite{pec2}. This follows directly from (\ref{E2}) because, by the
arithmetic-geometric mean inequality, we have

\begin{equation*}
\left[ f\left( x\right) \right] ^{t}\left[ f\left( y\right) \right]
^{1-t}\leq tf\left( x\right) +\left( 1-t\right) f\left( y\right)
\end{equation*}%
for all $x,y\in I$ and $t\in \left[ 0,1\right] $.

For some results related to this classical results, (see\cite{dragomir1},%
\cite{dragomir2},\cite{set1},\cite{set2}$)$ and the references therein.
Dragomir and Mond \cite{dragomir3} proved the following Hermite-Hadamard
type inequalities for the $\log $-convex functions:

\begin{eqnarray}
f\left( \frac{a+b}{2}\right) &\leq &\exp \left[ \frac{1}{b-a}%
\int\limits_{a}^{b}\ln \left[ f\left( x\right) \right] dx\right]  \label{z2}
\\
&\leq &\frac{1}{b-a}\int\limits_{a}^{b}G\left( f\left( x\right) ,f\left(
a+b-x\right) \right) dx  \notag \\
&\leq &\frac{1}{b-a}\int\limits_{a}^{b}f\left( x\right) dx  \notag \\
&\leq &L\left( f\left( a\right) ,f\left( b\right) \right)  \notag \\
&\leq &\frac{f\left( a\right) +f\left( b\right) }{2},  \notag
\end{eqnarray}%
where $G\left( p,q\right) =\sqrt{pq}$ is the geometric mean and $L\left(
p,q\right) =\frac{p-q}{\ln p-\ln q}$ $\left( p\neq q\right) $ is the
logarithmic mean of the positive real numbers $p,q$ $\left( \text{for }p=q,%
\text{ we put }L\left( p,q\right) =p\right) $.

Let us consider a function $\varphi :[a,b]\rightarrow \lbrack a,b]$ where $%
[a,b]\subset \mathbb{R}$. Youness have defined the $\varphi $-convex
functions in \cite{youness}:

\begin{definition}
A function $f:[a,b]\rightarrow \mathbb{R}$ is said to be $\varphi $- convex
on $[a,b]$ if for every two points $x\in \lbrack a,b],y\in \lbrack a,b]$ and 
$t\in \lbrack 0,1]$ the following inequality holds:%
\begin{equation*}
f(t\varphi (x)+(1-t)\varphi (y))\leq tf(\varphi (x))+(1-t)f(\varphi (y)).
\end{equation*}
\end{definition}

In \cite{cristescu1}, Cristescu proved the followig results for the $\varphi 
$-convex functions

\begin{lemma}
\label{l} For $f:[a,b]\rightarrow \mathbb{R}$, the following statements are
equivalent:

(i) $f$ is $\varphi $-convex functions on $[a,b]$,

(ii) for every $x,y\in \lbrack a,b]$, the mapping $g:[0,1]\rightarrow 
\mathbb{R},\ g(t)=f(t\varphi (x)+(1-t)\varphi (y))$ is classically convex on 
$[0,1].$
\end{lemma}

Obviously, if function $\varphi $ is the identity, then the classical
convexity is obtained from the previous definition. Many properties of the $%
\varphi $-convex functions can be found, for instance, in \cite{cristescu}, 
\cite{cristescu1},\cite{youness}.

In this paper, we introduce the notion of $\log $-$\varphi $-convex
functions and we obtain a representation of $\log $-$\varphi $-convex.
Finally, a version of Hermite--Hadamard-type inequalities for $\log $-$%
\varphi $-convex functions is presented.

\section{Main Results{}}

Let us consider a $\varphi :[a,b]\rightarrow \lbrack a,b]$ where $%
[a,b]\subset \mathbb{R}$ and $I$ stands for a convex subset of $\mathbb{R}$
. We say that a function $f:I\rightarrow \mathbb{R}^{+}$ is a $\log $-$%
\varphi $-convex if

\begin{equation}
f(t\varphi (x)+(1-t)\varphi (y))\leq \left[ f(\varphi (x))\right] ^{t}\left[
f(\varphi (y))\right] ^{1-t}  \label{1}
\end{equation}%
for all $x,y\in I$ and $t\in \lbrack 0,1]$. We say that $f$ is a $\log $-$%
\varphi $-midconvex if \ (\ref{1}) is assumed only for $t=\frac{1}{2}$, that
is%
\begin{equation*}
f\left( \frac{\varphi (x)+\varphi (y)}{2}\right) \leq \sqrt{f(\varphi
(x))(f(\varphi (y))},\ \text{for }x,y\in I
\end{equation*}%
Obviously, if function $\varphi $ is the identity, then the classical
logarithmic convexity is obtained from (\ref{1}).

From the above definitions, we have%
\begin{eqnarray*}
f(t\varphi (x)+(1-t)\varphi (y)) &\leq &\left[ f(\varphi (x))\right] ^{t}%
\left[ f(\varphi (y))\right] ^{1-t} \\
&& \\
&\leq &tf(\varphi (x))+(1-t)f(\varphi (y)) \\
&& \\
&\leq &\max \left\{ f\left( \varphi (x)\right) ,f\left( \varphi (y)\right)
\right\} .
\end{eqnarray*}

\begin{lemma}
\label{z} For $f:[a,b]\rightarrow \mathbb{R}^{+}$, the following statements
are equivalent:

(i) $f$ is $\log $-$\varphi $-convex functions on $[a,b]$,

(ii) for every $x,y\in \lbrack a,b]$, the mapping%
\begin{equation*}
g:[0,1]\rightarrow \mathbb{R}^{+},\ g(t)=f(t\varphi (x)+(1-t)\varphi (y))
\end{equation*}%
is classically $\log $-convex on $[0,1].$
\end{lemma}

\begin{proof}
Let us consider two points $x,y\in \lbrack a,b],\ \lambda \in \lbrack 0,1]$
and $t_{1},t_{2}\in \lbrack 0,1].$ Then, we obtain%
\begin{eqnarray*}
&&g(\lambda t_{1}+(1-\lambda )t_{2}) \\
&& \\
&=&f(\left[ \lambda t_{1}+(1-\lambda )t_{2}\right] \varphi (x)+\left[
1-\lambda t_{1}-(1-\lambda )t_{2}\right] \varphi (y)) \\
&& \\
&=&f(\lambda \left[ t_{1}\varphi (x)+(1-t_{1})\varphi (y)\right] +(1-\lambda
)\left[ t_{2}\varphi (x)+(1-t_{2})\varphi (y)\right] ) \\
&& \\
&\leq &\left[ f(t_{1}\varphi (x)+(1-t_{1})\varphi (y))\right] ^{\lambda }%
\left[ f\left( t_{2}\varphi (x)+(1-t_{2})\varphi (y)\right) \right]
^{1-\lambda } \\
&& \\
&=&\left[ g(t_{1})\right] ^{\lambda }\left[ g(t_{2})\right] ^{1-\lambda }
\end{eqnarray*}%
which gives that $g$ is $\log $-convex function.

Conversely, if $g$ is $\log $-convex function for $x,y\in \lbrack a,b],\
\lambda \in \lbrack 0,1]$ and $t_{1}=1,t_{2}=0$, then we get%
\begin{eqnarray*}
f(\lambda \varphi (x)+(1-\lambda )\varphi (y)) &=&g(\lambda 1+(1-\lambda )0))
\\
&& \\
&\leq &\left[ g(1)\right] ^{\lambda }\left[ g(0)\right] ^{1-\lambda } \\
&& \\
&=&\left[ f(\varphi (x))\right] ^{\lambda }\left[ f(\varphi (y))\right]
^{1-\lambda }
\end{eqnarray*}%
which shows that $f$ is $\log $-$\varphi $-convex. This completes to proof.
\end{proof}

We give now  a new Hermite--Hadamard-type inequalities for $\log $-$\varphi $%
-convex functions:

\begin{theorem}
If $f:\left[ a,b\right] \rightarrow \mathbb{R}^{+}$ is $\log $-$\varphi $%
-convex for the continuous function $\varphi :\left[ a,b\right] \rightarrow %
\left[ a,b\right] ,$ then%
\begin{eqnarray}
f\left( \frac{\varphi (a)+\varphi (b)}{2}\right)  &\leq &\frac{1}{\varphi
(b)-\varphi (a)}\dint\limits_{\varphi (a)}^{\varphi (b)}G\left(
f(x),f(\varphi (a)+\varphi (b)-x)\right) dx  \label{s1} \\
&\leq &\frac{1}{\varphi (b)-\varphi (a)}\dint\limits_{\varphi (a)}^{\varphi
(b)}f(x)dx  \notag \\
&\leq &\frac{f(\varphi (b))-f(\varphi (a))}{\log f(\varphi (b))-\log
f(\varphi (a))}=L\left( f(\varphi (b)),f(\varphi (a))\right)   \notag \\
&\leq &\frac{f(\varphi (a))+f(\varphi (b))}{2}.  \notag
\end{eqnarray}
\end{theorem}

\begin{proof}
Since $f$ be $\log $--$\varphi $-convex functions, we have that for all $%
t\in \left[ 0,1\right] $%
\begin{eqnarray*}
f\left( \frac{\varphi (a)+\varphi (b)}{2}\right) &=&f\left( \frac{t\varphi
(a)+(1-t)\varphi (b)}{2}+\frac{(1-t)\varphi (a)+t\varphi (b)}{2}\right) \\
&& \\
&\leq &\sqrt{\left[ f(t\varphi (a)+(1-t)\varphi (b))\right] \left[
f((1-t)\varphi (a)+t\varphi (b))\right] }
\end{eqnarray*}%
Integrating the above inequality with respect to $t$ over $[0,1]$ and we
also use the substitution $x=(1-t)\varphi (a)+t\varphi (b)$, we obtain 
\begin{eqnarray*}
&&f\left( \frac{\varphi (a)+\varphi (b)}{2}\right) \\
&\leq &\dint\limits_{0}^{1}\sqrt{\left[ f(t\varphi (a)+(1-t)\varphi (b))%
\right] \left[ f((1-t)\varphi (a)+t\varphi (b))\right] }dt \\
&=&\frac{1}{\varphi (b)-\varphi (a)}\dint\limits_{\varphi (a)}^{\varphi (b)}%
\sqrt{f(x)f(\varphi (a)+\varphi (b)-x)}dx \\
&\leq &\frac{1}{\varphi (b)-\varphi (a)}\dint\limits_{\varphi (a)}^{\varphi
(b)}A\left( f(x),f(\varphi (a)+\varphi (b)-x)\right) dx
\end{eqnarray*}

and so for%
\begin{equation*}
\dint\limits_{\varphi (a)}^{\varphi (b)}f(x)dx=\dint\limits_{\varphi
(a)}^{\varphi (b)}f(\varphi (a)+\varphi (b)-x)dx
\end{equation*}%
\begin{eqnarray}
&&f\left( \frac{\varphi (a)+\varphi (b)}{2}\right)   \label{s2} \\
&\leq &\frac{1}{\varphi (b)-\varphi (a)}\dint\limits_{\varphi (a)}^{\varphi
(b)}G\left( f(x),f(\varphi (a)+\varphi (b)-x)\right) dx  \notag \\
&\leq &\frac{1}{\varphi (b)-\varphi (a)}\dint\limits_{\varphi (a)}^{\varphi
(b)}f(x)dx.  \notag
\end{eqnarray}%
From the $\log $-$\varphi $-convexity of $f$, we have%
\begin{eqnarray}
&&\frac{1}{\varphi (b)-\varphi (a)}\dint\limits_{\varphi (a)}^{\varphi
(b)}f(x)dx  \label{s3} \\
&&  \notag \\
&=&\dint\limits_{0}^{1}f\left( t\varphi (a)+(1-t)\varphi (b)\right) dt 
\notag \\
&&  \notag \\
&\leq &\dint\limits_{0}^{1}\left[ f(\varphi (a))\right] ^{t}\left[ f(\varphi
(b))\right] ^{1-t}dt  \notag \\
&&  \notag \\
&=&f(\varphi (b))\dint\limits_{0}^{1}\left[ \frac{f(\varphi (a))}{f(\varphi
(b))}\right] ^{t}dt  \notag \\
&&  \notag \\
&=&f(\varphi (b))\frac{1}{\log f(\varphi (a))-\log f(\varphi (b))}\left[ 
\frac{f(\varphi (a))}{f(\varphi (b))}-1\right]   \notag \\
&&  \notag \\
&=&\frac{f(\varphi (b))-f(\varphi (a))}{\log f(\varphi (b))-\log f(\varphi
(a))}=L\left( f(\varphi (b)),f(\varphi (a))\right)   \notag \\
&&  \notag \\
&\leq &\frac{f(\varphi (a))+f(\varphi (b))}{2}.  \notag
\end{eqnarray}%
Thus, from (\ref{s2}) and (\ref{s3}) we obtain required result (\ref{s1}).
This completes to proof.
\end{proof}

\begin{theorem}
If $f,g:[a,b]\rightarrow \mathbb{R}^{+}$ is $\log $-$\varphi $- convex for
the continuous function $\varphi :[a,b]\rightarrow \lbrack a,b],$ then%
\begin{eqnarray}
&&\frac{1}{\varphi (b)-\varphi (a)}\dint\limits_{\varphi (a)}^{\varphi
(b)}f\left( x\right) g\left( x\right) dx\leq L\left( f(\varphi (b))g(\varphi
(b)),f(\varphi (a))g(\varphi (a))\right)   \label{5} \\
&&  \notag \\
&\leq &\frac{1}{4}\left\{ \left( \left[ f(\varphi (b))\right] +\left[
f(\varphi (a))\right] \right) L(\left[ f(\varphi (b))\right] ,\left[
f(\varphi (a))\right] )\right\}   \notag \\
&&  \notag \\
&&+\frac{1}{4}\left\{ \left( \left[ g(\varphi (b))\right] +\left[ g(\varphi
(a))\right] \right) L(\left[ g(\varphi (b))\right] ,\left[ g(\varphi (a))%
\right] )\right\} .  \notag
\end{eqnarray}
\end{theorem}

\begin{proof}
Since $f$ and $g$ be $\log $--$\varphi $-convex functions, we have that for
all $t\in \left[ 0,1\right] $%
\begin{equation*}
f(t\varphi (a)+(1-t)\varphi (b))\leq \left[ f(\varphi (a))\right] ^{t}\left[
f(\varphi (b))\right] ^{1-t}
\end{equation*}%
and 
\begin{equation*}
g(t\varphi (a)+(1-t)\varphi (b))\leq \left[ g(\varphi (a))\right] ^{t}\left[
g(\varphi (b))\right] ^{1-t}.
\end{equation*}%
Thus, it follows that%
\begin{eqnarray*}
&&\frac{1}{\varphi (b)-\varphi (a)}\dint\limits_{\varphi (a)}^{\varphi
(b)}f\left( x\right) g\left( x\right) dx \\
&&\dint\limits_{0}^{1}f(t\varphi (a)+(1-t)\varphi (b))g(t\varphi
(a)+(1-t)\varphi (b))dt \\
&\leq &\dint\limits_{0}^{1}\left[ f(\varphi (a))\right] ^{t}\left[ f(\varphi
(b))\right] ^{1-t}\left[ g(\varphi (a))\right] ^{t}\left[ g(\varphi (b))%
\right] ^{1-t}dt \\
&=&f(\varphi (b))g(\varphi (b))\dint\limits_{0}^{1}\left[ \frac{f(\varphi
(a))g(\varphi (a))}{f(\varphi (b))g(\varphi (b))}\right] ^{t}dt \\
&=&\frac{f(\varphi (b))g(\varphi (b))}{\log f(\varphi (a))g(\varphi
(a))-\log f(\varphi (b))g(\varphi (b))}\left[ \frac{f(\varphi (a))g(\varphi
(a))}{f(\varphi (b))g(\varphi (b))}-1\right]  \\
&& \\
&=&\frac{f(\varphi (b))g(\varphi (b))-f(\varphi (a))g(\varphi (a))}{\log
f(\varphi (b))g(\varphi (b))-\log f(\varphi (a))g(\varphi (a))} \\
&& \\
&=&L\left( f(\varphi (b))g(\varphi (b)),f(\varphi (a))g(\varphi (a))\right) 
\\
&& \\
&\leq &\frac{1}{2}\dint\limits_{0}^{1}\left( \left[ f(t\varphi
(a)+(1-t)\varphi (b))\right] ^{2}+\left[ g(t\varphi (a)+(1-t)\varphi (b))%
\right] ^{2}\right) dt \\
&\leq &\frac{1}{2}\dint\limits_{0}^{1}\left( \left[ f(\varphi (a))\right]
^{2t}\left[ f(\varphi (b))\right] ^{2-2t}+\left[ g(\varphi (a))\right] ^{2t}%
\left[ g(\varphi (b))\right] ^{2-2t}\right) dt \\
&=&\frac{1}{4}\left\{ \left[ f(\varphi (b))\right] ^{2}\dint\limits_{0}^{2}%
\left[ \frac{f(\varphi (a))}{f(\varphi (b))}\right] ^{u}du+\left[ g(\varphi
(b))\right] ^{2}\dint\limits_{0}^{2}\left[ \frac{g(\varphi (a))}{g(\varphi
(b))}\right] ^{u}du\right\}  \\
&=&\frac{1}{4}\left\{ \frac{\left[ f(\varphi (b))\right] ^{2}-\left[
f(\varphi (a))\right] ^{2}}{\log f(\varphi (b))-\log f(\varphi (a))}+\frac{%
\left[ g(\varphi (b))\right] ^{2}-\left[ g(\varphi (a))\right] ^{2}}{\log
g(\varphi (b))-\log g(\varphi (a))}\right\}  \\
&=&\frac{1}{4}\left\{ \left( \left[ f(\varphi (b))\right] +\left[ f(\varphi
(a))\right] \right) L(\left[ f(\varphi (b))\right] ,\left[ f(\varphi (a))%
\right] )\right\}  \\
&&+\frac{1}{4}\left\{ \left( \left[ g(\varphi (b))\right] +\left[ g(\varphi
(a))\right] \right) L(\left[ g(\varphi (b))\right] ,\left[ g(\varphi (a))%
\right] )\right\} 
\end{eqnarray*}%
which is the required (\ref{5}). This proves the theorem.
\end{proof}

\end{document}